\documentclass{amsart}

\usepackage{amssymb,latexsym}
\usepackage{enumerate}

\theoremstyle{plain}
\newtheorem{theorem}{Theorem}[section]
\newtheorem{lemma}[theorem]{Lemma}
\newtheorem{corollary}[theorem]{Corollary}
\newtheorem{proposition}[theorem]{Proposition}

\newtheorem*{step1}{Step 1}
\newtheorem*{step2}{Step 2}
\newtheorem*{step3}{Step 3}

\theoremstyle{definition}

\newcommand{\lset}{\{}
\newcommand{\rset}{\}}
\newcommand{\inv}{^{-1}}
\newcommand{\supe}{\supseteq}
\newcommand{\sube}{\subseteq}
\newcommand{\mdot}{{\cdot}}

\newcommand{\ZZ}{\mathbb{Z}}
\newcommand{\QQ}{\mathbb{Q}}
\newcommand{\CC}{\mathbb{C}}

\newcommand{\us}{^{*}}

\newcommand{\vf}{\varphi}
\newcommand{\vt}{\vartheta}
\newcommand{\ve}{\varepsilon}

\newcommand{\ch}{\operatorname{char}}
\newcommand{\Imm}{\operatorname{Im}}
\newcommand{\nor}{\triangleleft}
\newcommand{\tr}{\operatorname{tr}}

\makeatletter
\newcommand{\cmpl}[1]{%
    \sbox\z@{$#1$}%
    \dimen@=\wd\z@
    \advance \dimen@ -\strip@pt\fontdimen\@ne\textfont\@ne \ht\z@
    \setbox\tw@=\hb@xt@\dimen@{}%
    \ht\tw@=\ht\z@ \dp\tw@=\dp\z@
    \box\z@
    \llap{$\overline{\box\tw@}$}%
} 
\makeatother

\newcommand{\cx}{\cmpl{x}}
\newcommand{\cy}{\cmpl{y}}
\newcommand{\cz}{\cmpl{z}}
\newcommand{\cp}{\cmpl{\varphi}}
\newcommand{\fg}{\mathfrak{fg}}
\newcommand{\la}{\langle}
\newcommand{\ra}{\rangle}
\newcommand{\fl}{\mathfrak{fl}}
\newcommand{\ad}{\mathrm{ad}\,}

\begin{document}
\title{Finite Presentation}

\author{D.~S.~Passman}
\address[D.~S.~Passman]{Department of Mathematics, University of Wisconsin-Madison, Madison,
Wisconsin 53706, USA} 
\email{passman@math.wisc.edu}

\author{L.~W.~Small}
\address[L.~W.~Small]{Department of Mathematics, University of California-San Diego, La Jolla, California 92093, USA} 
\email{lwsmall@ucsd.edu}

\subjclass[2010]{}

\begin{abstract}
This paper surveys basic properties of finite presentation in groups,
Lie algebras and rings. It includes some new results and also
new, more elementary proofs, of some results that are already in the literature. 
In particular, we discuss examples of Stallings and of Roos 
on coherence by using a
purely algebraic, non-homological, approach.
\end{abstract}

\maketitle

\section{Introduction}

In this paper we study finitely presented algebraic objects in four different contexts. Admittedly, we could work with more
general universal algebras, but we prefer to keep things concrete.  The objects of interest here are:
\begin{enumerate}
\item groups,
\item Lie algebras over a field $K$,
\item $R$-algebras where $R$ is a ring, and
\item more general rings.
\end{enumerate}
In case (iii), we mean rings $S$ that contain $R$ as a subring with the same 1 and are
generated by $R$ and its centralizer $\CC_S(R)$ in $S$.
Now with each of these there are free objects, namely free groups, free Lie algebras, free $R$-algebras and free rings.
Furthermore, there are homomorphisms and their kernels. Of course, kernels for groups are normal subgroups, while the others are
Lie ideals and just plain ideals.

We say that one of these objects $A$ is \emph{finitely presented} if there exists a finitely generated free object $F$
and an epimorphism $\pi\colon F\to A$ such that the kernel of $\pi$ is finitely generated as a kernel, that is as either a
normal subgroup or an ideal. Note that a finitely presented object is necessarily finitely generated. Some basic results are
given below. 

\begin{lemma}
Let $A$ and $B$ be finitely generated objects of the same type and let $A$ be finitely presented. If
$\vt\colon B\to A$ is an epimorphism, then the kernel of $\vt$ is finitely generated as a kernel.
\end{lemma}

\begin{proof}
Since $A$ is finitely presented, there exists a free object $F=\langle x_1,x_2,\ldots, x_n\rangle$, with the $x_i$ as free
generators, and an epimorphism $\pi\colon F\to A$, such that $\ker\pi=J$ is finitely generated as a kernel. Clearly $A$ is
generated by the elements $a_i=\pi(x_i)$. Next, we are given the epimorphism $\vt\colon B\to A$, so we can choose
$b_i\in B$ with $\vt(b_i)=a_i$. Since $F$ is free on the $x_i$'s, we can define a homomorphism $\alpha\colon F\to B$
by $\alpha(x_i)=b_i$. Notice that the composition $\vt\alpha\colon F\to B\to A$ is equal to $\pi$.
In particular, if $f\in F$, then $\alpha(f)\in\ker\vt$ if and only if $f\in\ker\pi$.

Now $B=\langle g_1,g_2,\ldots, g_m\rangle$ is finitely generated, and since $\pi$ is onto, we can choose $f_j\in F$
with $\vt(g_j)=\pi(f_j) =\vt\alpha(f_j)$. Thus $g_j-\alpha(f_j)\in \ker\vt$ for all $j$. Of course, in the context of groups,
addition corresponds to multiplication and subtraction to a version of division. Let $I$ be the kernel in $B$
generated by the finitely many elements $g_j-\alpha(f_j)$. Since each $g_j$ is contained in $I+\Imm(\alpha)$, and
since $I$ is a kernel, we see that $B=I+\Imm(\alpha)$.

Finally, let $u\in\ker\vt$. Then $u\in I+\Imm(\alpha)$, so $u=v+\alpha(f)$ for some $v\in I$ and $f\in F$.
Since $v\in I\sube\ker\vt$, we have $\alpha(f)\in\ker\vt$ and hence $f\in\ker\pi=J$. Thus since $\alpha(J)\sube\ker\vt$,
we see that $\ker\vt=I+\alpha(J)$. Now $I$ is finitely generated as a kernel and $J$ is finitely generated as
a kernel in $F$. Thus $\alpha(J)$ is finitely generated as a kernel in $\alpha(F)\sube B$. We conclude that
$\ker\vt=I+\alpha(J)$ is finitely generated as a kernel in $B$, and the result follows.
\end{proof}

The above efficient proof is based on the one in \cite{RS}. Conversely we have

\begin{lemma}
Let $A$ and $B$ be finitely generated objects of the same type and let $B$ be finitely presented. If $\vt\colon B\to A$
is an epimorphism and if the kernel of $\vt$ is finitely generated as a kernel, then $A$ is finitely presented.
\end{lemma}

\begin{proof}
Since $B$ is finitely presented, there exists a finitely generated free object $F$ and an epimorphism
$\pi\colon F\to B$ so that the kernel of $\pi$ is finitely generated as a kernel. Then the composite map
$\vt \pi\colon F\to B\to A$ is an epimorphism and its kernel $L$ in $F$ satisfies $L\supe \ker \pi$ and
$B = F/\ker \pi \supe L/\ker\pi =\ker \vt$. Since $\ker \vt$ is finitely generated as a kernel in $B$ and $\ker\pi$
is finitely generated as a kernel in $F$, it is clear that $L$ is finitely generated as a kernel in $F$. Hence $A$
is finitely presented.
\end{proof}

As a consequence, we obtain

\begin{lemma}
Let $A$ be a finitely generated object and write $A=B\oplus C$, a direct sum of objects of the same type. Then $B$ and $C$ are finitely generated.
Furthermore, if $A$ is finitely 
presented, then so are $B$ and $C$.
\end{lemma}

\begin{proof}
We have an epimorphism $\pi\colon A\to B$ with kernel $C$ and we have an epimorphism $\vt\colon A\to C$.
Since $A$ is finitely generated, the latter implies that $C$ is finitely generated and hence finitely generated as a kernel.
The previous lemma now yields the result.

Note that if $A$ is a $R$-algebra or a ring, then $C$ is an ideal of $A$ and it is singly generated as an ideal by the
central idempotent $e = 0\oplus 1$, that is $C=Ae$.
\end{proof}

Furthermore, we have

\begin{proposition}
Let $A$ be a finitely presented object. Then all homomorphic images of $A$ are
finitely presented if and only if $A$ satisfies the ascending chain condition on kernels.
\end{proposition}

\begin{proof}
Of course $A$ satisfies the ascending chain condition on kernels if and only if
all kernels are finitely generated as kernels. Thus Lemmas 1.1 and 1.2 yield the result.
\end{proof}

Recall that an object $A$ is said to be \emph{Hopfian} if every epimorphism $A\to A$ is one-to-one
and hence an isomorphism.  If $A$ satisfies the ascending chain condition on kernels, then it is easy
to see that $A$ is Hopfian. Indeed, if $N$ is maximal with $A/N \cong A$, then $A/N$ can have no proper
homomorphic image isomorphic to $A$, and hence $A/N\cong A$ is Hopfian. As a consequence, we have

\begin{corollary}
If all homomorphic images of $A$ are finitely presented, then $A$ is Hopfian.
\end{corollary}

We close this section with three lemmas that offer some converses and analogs to Lemma~1.3 for our specific objects of interest. In all of these situations, we will have
$A=\la B,C\ra$ so that $A$ is somehow generated by $B$ and $C$. If $B$
is finitely generated by the generator set $\beta$ and $C$ is finitely
generated by the generator set $\gamma$, then surely $A$ is finitely
generated by $\beta\cup\gamma$. Furthermore, if $B$ and $C$ are finitely
presented, then $\beta\cup \gamma$ satisfies the finitely many relations
of $\beta$ and of $\gamma$. Furthermore, depending on the particular
structure of $A$, we must adjoin additional relations that are
mentioned in each of the three lemmas. We will use this notation and observation
below. We start with groups.

\begin{lemma}
Let $B$ and $C$ be multiplicative groups.
\begin{enumerate}
\item $B$ and $C$ are finitely generated (respectively, finitely presented)
if and only if their direct product $A=B\times C$ is finitely generated
(respectively, finitely presented).
\item $B$ and $C$ are finitely generated (respectively, finitely presented)
if and only if their free product $A=B* C$ is finitely generated
(respectively, finitely presented).
\end{enumerate}
\end{lemma}

\begin{proof}
(i) We know from Lemma 1.3 that if $A$ is finitely generated or finitely
presented, then so are $B$ and $C$. Conversely if $B$ and $C$ are finitely 
generated, then so is $A=\la B,C\ra$. Furthermore, if $B$ and $C$ are
finitely presented, then we must adjoin to the relations for $\beta$
and $\gamma$, the finitely many relations that assert that each element
of $\beta$ commutes with each element of $\gamma$.

(ii)  If $B$ and $C$ are finitely generated, then so is $A=\la B,C\ra$.
Furthermore, if $B$ and $C$ are
finitely presented, then freeness implies that no additional
relations need be adjoined to the finite number of relations for $\beta$
and for $\gamma$. Thus $A=B*C$ is finitely presented.

Conversely, note that we have epimorphisms
$\theta_B\colon A \to B$ given by $B*C\to B*1\cong B$ and
$\theta_C\colon A\to C$. Thus if $A$ is finitely generated, then so are $B$
and $C$. Furthermore, if $A$ is finitely presented, then since the
kernel of $\theta_B$ is generated by $C$ and hence by the finite set $\gamma$,
we conclude from Lemma~1.2 that $B$ is finitely presented. Similarly,
$C$ is also finitely presented.
\end{proof}

The proof of the next lemma is identical to that of part (i) above.

\begin{lemma}
Let $B$ and $C$ be Lie algebras over the field $K$.
Then $B$ and $C$ are finitely generated (respectively, finitely presented)
if and only if their direct sum $A=B\oplus C$ is finitely generated
(respectively, finitely presented).
\end{lemma}

Finally, we consider associative algebras.

\begin{lemma}
Let $B$ and $C$ be associative $K$-algebras.
\begin{enumerate}
\item $B$ and $C$ are finitely generated (respectively, finitely presented)
if and only if their direct sum $A=B\oplus C$ is finitely generated
(respectively, finitely presented).
\item $B$ and $C$ are finitely generated (respectively, finitely presented)
if and only if their tensor product $A=B\otimes_K C$ is finitely generated
(respectively, finitely presented).
\end{enumerate}
\end{lemma}

\begin{proof}
(i) Again, we know from Lemma 1.3 that if $A=B\oplus C$ is finitely generated or finitely
presented then so are $B$ and $C$. Conversely if $B$ and $C$ are finitely 
generated, then so is $A=\la B,C\ra$. Furthermore, suppose $B$ and $C$ are
finitely presented and let $1_B$ and $1_C$ be their 
corresponding identity elements. Then we adjoin to the finite number of relations for $\beta$
and $\gamma$, the relation $1_B  1_C = 1_C 1_B=0$. In this way,
$1_B$ and $1_C$ become orthogonal idempotents whose sum clearly acts like
the identity on $A$. Thus $A\cong B\oplus C$ is finitely presented.

(ii) If $B$ and $C$ are finitely 
generated, then so is $A=\la B,C\ra$. Furthermore, if $B$ and $C$ are
finitely presented, then we must adjoin to the relations for $\beta$
and $\gamma$, the finitely many relations that assert that each element
of $\beta$ commutes with each element of $\gamma$. In this way,
$A=B\otimes C$ is finitely presented.

Conversely, suppose $A=B\otimes C$ is finitely generated. Then $A$ has
countable dimension over $K$, and hence so does $B$.
The latter implies that we can write $B$ as an ascending union
$B=\bigcup_m B_m$ of finitely generated subalgebras.
Hence $A$ is the ascending union of the various $B_m\otimes C$. But $A$
is finitely generated, so this union must terminate in a finite
number of steps. Thus $B\otimes C=A= B_m\otimes C$ for some $m$
and $B=B_m$ is finitely generated. Similarly, $C$ is a finitely generated
algebra.

Finally, suppose $A = B\otimes C$ is finitely presented. Then $A$ and $B$
are finitely generated, so we can map a finitely generated free algebra
$F$ onto $B$ with kernel $I$. Then $F\otimes C$ maps onto $B\otimes C=A$
with kernel $I\otimes C$, and by Lemma~1.1, $I\otimes C$ is finitely
generated as an ideal.
Now $I$ has countable dimension over $K$,
so we can write $I$ as an ascending union $\bigcup_n I_n$ of ideals $I_n$,
each of which is finitely generated as an ideal. Hence $I\otimes C$ is the
ascending union of the various $I_n\otimes C$. But $I\otimes C$ is
finitely generated as an ideal, so this union necessarily terminates in a finite
number of steps. Thus $I\otimes C=I_n\otimes C$ for some $n$,
and hence $I=I_n$ is finitely generated as an ideal of $F$.
It follows that $B\cong F/I$ is finitely presented, and similarly so is $C$.
\end{proof}

\section{Groups and their Group Rings}

Let $G$ be a multiplicative group and let $R$ be a ring. Then the group ring $R[G]$ is clearly an $R$-algebra
as was defined in the previous section. Notice that if $N\nor G$, then the epimorphism $G\to G/N$ extends to
an epimorphism $\vf_N\colon R[G]\to R[G/N]$. In particular, when $N=G$ we have an epimorphism
$\vf_G\colon R[G]\to R$ sending each group element to 1. This is known as the \emph{augmentation map}
and its kernel is the \emph{augmentation ideal} $\omega(R[G])$. Clearly the latter is the set of group ring
elements with coefficient sum 0 and hence it is the $R$-linear span of the elements $1-g$ with $g\in G$.
More generally, the kernel of $\vf_N$ is easily seen to be the two-sided ideal
$\omega(R[N])\mdot R[G]$ and this is generated by all $1-g$ with $g\in N$.

\begin{lemma}
$N$ is finitely generated as a normal subgroup of $G$ if and only if $\omega(R[N])\mdot R[G]$ is finitely
generated as an ideal of $R[G]$.
\end{lemma}

\begin{proof}
First suppose $N$ is generated as a normal subgroup by $g_1, g_2,\ldots, g_n\in N$ and let $I$
be the ideal of $R[G]$ generated by the elements $1-g_i$. Then certainly $I\sube \omega(R[N])\mdot R[G]$.
Conversely, we have a group homomorphism $\alpha$ from $G\sube R[G]$ into the group of units of $R[G]/I$
and notice that $g_1,g_2,\ldots, g_n$ are in the kernel of this map. Thus since the kernel of $\alpha$ is a normal subgroup
of $G$, we must have $N\sube\ker\alpha$. Thus $\omega(R[N])\sube I$ and $I=\omega(R[N])\mdot R[G]$
is finitely generated by the elements $1-g_i$.

In the other direction, let $\omega(R[N])\mdot R[G]$ be finitely generated as an ideal. Since this ideal is generated by 
the various $1-g$ with $g\in N$, it follows that it is finitely generated by $1-g_1,1-g_2,\ldots, 1-g_n$ for some $g_i\in N$.
Now let $M$ be the normal subgroup of $G$ generated by $g_1,g_2,\ldots, g_n$. Then $M$ is finitely generated as a normal
subgroup of $G$, $M\sube N$, and each $1-g_i$ is contained in $\omega(R[M])\mdot R[G]\sube \omega(R[N])\mdot R[G]$.
But the elements $1-g_i$ generate the larger ideal, so we must have $\omega(R[M])\mdot R[G]=\omega(R[N])\mdot R[G]$,
and hence $N=M$ is finitely generated as a normal subgroup.
\end{proof}

Now we need one particular example of interest.

\begin{lemma}
Let $G=\langle g_1,g_2,\ldots, g_n\rangle$ be the free group on the $n$ free generators $g_1,g_2,\ldots, g_n$,
and let $R$ be any ring. Then the group ring $R[G]$ is finitely presented as an $R$-algebra.
\end{lemma}

\begin{proof}
Let $F=R\langle x_1,x_2,\ldots,x_n, y_1, y_2,\ldots,y_n\rangle$ be the free $R$-algebra on the $2n$ free generators
$x_i, y_i$ that commute with $R$. Then there exists an $R$-epimorphism $\pi\colon F\to R[G]$ given by
$\pi(x_i)=g_i$ and $\pi(y_i)=g_i\inv$. The kernel of $\pi$ clearly contains the $2n$ elements $x_iy_i-1$ and
$y_ix_i-1$. Now let $I$ be the ideal of $F$ generated by the various elements $x_iy_i-1$ and $y_ix_i-1$,
and let $\cmpl{\phantom{x}}$ denote the epimorphism $F\to F/I$. Then $\pi$ factors through $\cmpl{\phantom{x}}$ 
so there exists $\cmpl{\pi}\colon \cmpl{F}\to R[G]$ with $\cmpl{\pi}(\cmpl{x}_i)=g_i$ and
$\cmpl{\pi}(\cmpl{y}_i)=g_i\inv$. On the other hand, note that $\cmpl{x}_i$ is invertible in $\cmpl{F}$ with inverse
$\cmpl{y}_i$, so since $G$ is free, there is a homomorphism $\vt$ from $G$ to the units of $\cmpl{F}$
with $\vt(g_i)=\cmpl{x}_i$ and $\vt(g_i\inv)=\cmpl{y}_i$. Of course, $\vt$ extends to an epimorphism
$\vt\colon R[G]\to \cmpl{F}$. Since $\cmpl{\pi}$ and $\vt$ are clearly inverses of each other, we see that both
are one-to-one and hence $I=\ker\pi$. Thus $\ker\pi$ is finitely generated as an ideal, and $R[G]$ is finitely
presented as an $R$-algebra.
\end{proof}

The following can be found in Baumslag \cite{Bm} with a proof that is perhaps a bit too skimpy.

\begin{theorem}
Let $G$ be a multiplicative group and let $R$ be a ring. Then $G$ is finitely presented as a group
if and only if its group ring $R[G]$ is finitely presented as an $R$-algebra.
\end{theorem}

\begin{proof}
Suppose first that $G$ is finitely presented. Then there exists a finitely generated free group $F$
and an epimorphism $\pi\colon F\to G$ such that $N=\ker\pi$ is finitely generated as a normal subgroup of $F$.
By Lemma~2.1, the corresponding epimorphism $\pi\colon R[F]\to R[F/N]\cong R[G]$ has a kernel
that is finitely generated as an ideal. Thus since $R[F]$ is finitely presented by the previous lemma, we conclude from Lemma~1.2
that $R[G]$ is finitely presented.

Conversely suppose $R[G]$ is finitely presented as an $R$-algebra. Then $R[G]$ is finitely generated over $R$ and hence
$G$ is finitely generated by the supports of the finite number of generators of $R[G]$. In particular, there exists
a finitely generated free group $F$ with $G\cong F/N$ for some normal subgroup $N$. Since $R[G]$ is finitely
presented, the kernel of the epimorphism $R[F]\to R[F/N]\cong R[G]$ is finitely generated as an ideal by Lemma 1.1
and hence $N$ is finitely generated as a normal subgroup by Lemma~2.1. Thus $G$ is indeed finitely presented.
\end{proof}

The proof of the next result uses a simplification of a group construction due to
Abels \cite{A}.

\begin{lemma}
Let $G$ be a finitely generated infinite group. Then the wreath product $\mathbb{Z}\wr G$
is a finitely generated group that is not finitely presented.
\end{lemma}

\begin{proof}
Let $S=\mathbb{Z}[G]$ be the integral group ring of $G$.
We consider the group $\mathcal{G}$ of $3\times 3$ matrices
over $S$ of the form
\[[g,a,b,c]=\begin{bmatrix} 1 & a & c\\ & g & b\\ & & 1 \end{bmatrix} \]
with $g\in G$ and $a,b,c\in S$. It can be shown that this group is finitely generated
with an infinitely generated center. However we are actually concerned with a  more interesting,
somewhat smaller subgroup of $\mathcal{G}$.

To start with, for each $g\in G$, write $\bar g=[g,0,0,0]$ and let $\cmpl G=\lset \bar g\mid g\in G\rset$.
Then the map $G\to \cmpl G$ given by $g\mapsto \bar g$ is clearly an isomorphism, so
$\cmpl G$ is a finitely generated subgroup of $\mathcal{G}$. Next, note that the map
$\mathcal{G}\to G$ given by $[g,a,b,c] \mapsto g$ is an epimorphism with kernel
$\mathcal{H}$, the set of all matrices of the form $[1,a,b,c]$
with  $a,b,c\in S$. Thus $\mathcal{H}$ is a normal subgroup of $\mathcal{G}$ and clearly
$\mathcal{G}= \mathcal{H}\rtimes \cmpl{G}$, the semidirect product of $\mathcal{H}$
by $\cmpl G$. Of course, the set of matrices $\mathcal{Z}\sube \mathcal{G}$
of the form $[1,0,0,c]$
with  $c\in S$ is central in $\mathcal{G}$ and is isomorphic to the additive subgroup of $S$.
Thus the center of $\mathcal{G}$ has an infinitely generated torsion-free subgroup
and hence is not finitely generated.

Now let ${}\us$ be the classical antiautomorphism of $S=\ZZ[G]$ determined by $g\us=g\inv$
for all $g\in G$, and let $\mathcal{H}\us$ be the subset of $\mathcal{H}$ given by all
elements $[1,a,b,c]$ with $b=a\us$. Since we have
\begin{equation} \tag{{\bf mult}}
[1,a,a\us,c] [1,b,b\us,d] =[1,a+b, a\us+b\us, c+d+ab\us]
\end{equation}
it follows easily that $\mathcal{H}\us$ is closed under multiplication and inverses. Hence
$\mathcal{H}\us$ is a subgroup of $\mathcal{H}$. Furthermore, $\mathcal{H}\us\supe \mathcal{Z}$
and $\mathcal{H}\us/\mathcal{Z}$ is isomorphic to the additive subgroup of $S$ via the map
$[1,a,a\us,c]\mapsto a$. Next, for any $g\in G$, we have
\begin{equation}\tag{{\bf conj}}
{\bar g}\inv [1,a,a\us,c] \bar g = [1,ag, g\inv a\us,c] =[1,ag, (ag)\us,c]\end{equation}
so $\cmpl G$ normalizes $\mathcal{H}\us$. Thus $\mathcal{G}\us=
\mathcal{H}\us \cmpl{G}\cong \mathcal{H}\us \rtimes G$ is a
subgroup of $\mathcal{G}$.

Finally, let $\mathcal{W}$ be the finitely generated subgroup of $\mathcal{G}\us$ generated by
$\cmpl{G}$ and $[1,1,1,0]\in\mathcal{H}\us$. Since ${\bar g}\inv [1,1,1,0] g=[1,g,g\us,0]$, it follows
from equation ({\bf mult}) that for all $a\in S$, there exists some $c\in S$, depending on $a$,
with $[1,a,a\us,c]\in\mathcal{W}$.  We can also use equation ({\bf mult}) to compute the commutator
of  two elements of $\mathcal{W}$, say $[1,a,a\us,c]$ and $[1,b,b\us,d]$. Specifically, this
commutator $[1,0,0,r]$ satisfies
\[ [1,a,a\us,c] [1,b,b\us,d]= [1,b, b\us,d][1,a,a\us,c][1,0,0,r]\]
and hence
\begin{align*}
[1,a+b, a\us +b\us, c+d+ab\us]&=[1,a+b,a\us+ b\us, c+d+ba\us][1,0,0,r] \\
&= [1,a+b,a\us+ b\us, c+d+r+ba\us].
\end{align*}
Thus $r= ab\us-b a\us$, and by taking $b=1$, we see that $[1,0,0,a-a\us]\in \mathcal{W}\cap\mathcal{Z}$
for all $a\in S$.

Since $G$ is infinite, it follows that $\mathcal{W}\cap \mathcal{Z}$ is an infinitely generated
central subgroup of $\mathcal{W}$. Thus, since $\mathcal{W}$ is finitely generated, Lemma~1.1
implies that the group $\mathcal{W}/(\mathcal{W}\cap \mathcal{Z})$ is finitely generated, but not
finitely presented. It remains to understand the latter factor group. To this end, note that
$\mathcal{W}\mathcal{Z}=\mathcal{G}\us$, so
\[ \mathcal{W}/(\mathcal{W}\cap \mathcal{Z})\cong \mathcal{G}\us/\mathcal{Z} \cong (\mathcal{H} \us/\mathcal{Z})
\rtimes \cmpl G\cong (\mathcal{H} \us/\mathcal{Z})\rtimes G.\] 
Furthermore, $\mathcal{H} \us/\mathcal{Z}$ is isomorphic to the additive group of $S$, and $G$
acts on $S$ via right multiplication. Thus $\mathcal{W}/(\mathcal{W}\cap \mathcal{Z})\cong S\rtimes G$.
Now, $S$ additively is the free abelian group with $\mathbb{Z}$-basis $G$, and $G$ acts on $S$ by regularly permuting this basis.
By definition, this means that $S\rtimes G$ is isomorphic to the wreath product $\mathbb{Z}\wr G$, and therefore the
lemma is proved.
\end{proof}

For example, in the above lemma, we can take $G=\mathbb{Z}$ to be infinite cyclic. Then 
we conclude that $\ZZ\wr \ZZ$
is finitely generated but not finitely presented. This group is of course the semidirect product
of the free abelian group $A$, with free generators $\lset a_i\mid i\in\ZZ\rset$, by the infinite cyclic group 
$\langle x\rangle$, with $x\inv a_i x= a_{i+1}$ for all $i$.

Next, note that there are numerous examples of finitely presented groups that are not
Hopfian. Most well-known are the Baumslag-Solitar groups
\[ \mathrm{BS}(m,n) =\langle a,b \mid a\inv b^m a = b^n\rangle\]
where the parameters $m$ and $n$ are of course nonzero integers.
Indeed, according to \cite[Theorem 1]{BS}, $\mathrm{BS}(m,n)$ is
Hopfian if and only if $m$ or $n$ divides the other, or
$m$ and $n$ have precisely the same prime divisors.

In particular, the group
\[ G = \mathrm{BS}(2,3) =\langle a,b \mid a\inv b^2 a = b^3\rangle\]
is non-Hopfian. To see this, note that the subgroup of $G$
generated by $a$ and $b^2$ contains $a\inv b^2 a=b^3$ and hence $b$,
so $\langle a, b^2\rangle =\langle a,b\rangle= G$. Furthermore, by squaring both sides of
the relation $a\inv b^2 a = b^3$, we get $a\inv (b^2)^2 a = (b^2)^3$,
and thus there exists an epimorphism $\pi\colon G\to G$
given by $\pi(a)=a$ and $\pi(b)=b^2$. Finally, it follows from work
of Magnus \cite{M} that $a\inv b a$ does not commute with $b$ in $G$.
In particular, the commutator $[a\inv ba,b]$ is not 1. But
\[\pi([a\inv ba,b])= [a\inv b^2a,b^2]= [b^3,b^2]=1,\]
so $[a\inv ba,b]$ is a nonidentity element in $\ker\pi$.

We can of course use the group $ G = \mathrm{BS}(2,3)$ to obtain
a finitely presented, non-Hopfian group ring. Indeed, let $R$ be any ring with 1.
Then by Theorem 2.3, $R[G]$ is a finitely presented group ring.
Furthermore, the group epimorphism $\pi\colon G\to G$ extends to a group ring
epimorphism $\pi'\colon R[G]\to R[G]$. Since $0\neq 1-[a\inv ba,b]$ is in the
kernel of $\pi'$, we conclude that $R[G]$ is non-Hopfian.

A group is said to be \emph{coherent} if every finitely generated subgroup
is finitely presented. Certainly, every free group is coherent.
The following example, due to Stallings in \cite{Sta},
shows that finitely presented groups are not necessarily coherent. We offer a group theoretic proof here rather than the original topological one. For convenience, we write $\mathfrak{fg}\langle x,y\rangle$ for the free group on the two generators.

\begin{theorem}
Let $T=\mathfrak{fg}\langle a,b\rangle \times \mathfrak{fg}\langle c,d\rangle $ be the direct product of the two free groups
and let $S$ be the three generator subgroup of $T$
given by $S=\langle a,bc,d\rangle$.
Then $S\nor T$ with $T/S$ infinite cyclic. Furthermore, $T$ is finitely presented but $S$ is not. In particular, $T$ is not coherent.
\end{theorem}

\begin{proof}
Since free groups are finitely presented, we know that $T$ is
finitely presented. We now 
proceed in a series of three steps.
\end{proof}

\begin{step1} The structure of $S$.
\end{step1}

\begin{proof}
As above, we write $\mathfrak{fg}\langle x,y\rangle $ for the free group on the two generators.
Then we have an epimomorphism $\mathfrak{fg}\langle x,y\rangle  \to \langle y\rangle$
given by $x\mapsto 1$, $y\mapsto y$ and we denote its kernel by 
$\mathfrak{fg}\langle x,y\rangle_x$. This kernel is the normal subgroup of $\mathfrak{fg}\langle x,y\rangle $
generated as a subgroup by all conjugates $x^{(y^n)}$ for $n\in \mathbb{Z}$.
Indeed, the subgroup of $\fg\la x,y\ra$ generated by all 
$y^n$-conjugates of $x$ is clearly normalized by $x$ and $y$, and when one mods
out by this subgroup, only $y$ remains.
Similarly, we write $\mathfrak{fg}\langle x,y\rangle_y$ for the subgroup of 
$\mathfrak{fg}\langle x,y\rangle $ generated by all conjugates $y^{(x^n)}$ with 
$n\in\mathbb{Z}$.  

Now the projection of $S$ onto the first factor of $T$ is given
by $a\mapsto a$, $bc\mapsto b$ and $d\mapsto 1$. Thus
$\langle a,bc\rangle$ is free of rank 2, and similarly $\langle bc,d\rangle$
is also free of rank 2. Since $c$ commutes with $a$ and $b$, we see that
$a^{((bc)^n)} = a^{(b^n)}$ and hence $\mathfrak{fg}\langle a,bc\rangle_a =
\mathfrak{fg}\langle a,b\rangle_a\sube \mathfrak{fg}\langle a,b\rangle$. In other words, the latter two ``sub $a$'' groups are
identical as subgroups of $T$. Note that the first formulation 
shows that the subgroup is in $S$
and is normalized by $a$ and $bc$, while the second formulation 
shows that it is centralized
by $c$ and $d$. Thus this group is normal in $T$. Similarly we have
$\mathfrak{fg}\langle bc,d\rangle_d = \mathfrak{fg}\langle c,d\rangle_d\sube \mathfrak{fg}\langle c,d\rangle$
is normal in $T$.
Set $S_0 = \mathfrak{fg}\langle a,b\rangle_a\times \mathfrak{fg}\langle c,d\rangle_d\sube 
\mathfrak{fg}\langle a,b\rangle\times \mathfrak{fg}\langle c,d\rangle$ so $S_0\sube S$, $S_0\nor T$
and $T/S_0$ is naturally isomorphic to the free abelian group
$\langle b,c\rangle$. Furthermore, $S/S_0$ corresponds to
the subgroup $\langle bc\rangle$, so $S/S_0$ is infinite cyclic. 
In addition, $S\nor T$ and $T/S$ is
infinite cyclic.
\end{proof}

\begin{step2}
The relations of $S$.
\end{step2}

\begin{proof}
Let $F=\mathfrak{fg}\langle x,y,z\rangle$ be the free group on generators $x,y,z$
and consider the epimorphism $\varphi\colon F\to S$ given by
$\varphi(x)=a,\ \varphi(y)=bc$ and $\varphi(z)=d$. We will precisely determine
the kernel of $\varphi$. To this end, let $N$ be the normal subgroup
of $F$ generated, as a normal subgroup, by the commutators
\[ g(m,n)=[ x^{(y^m)}, z^{(y^n)} ] \qquad \mathrm{for\ all}\quad m,n\in\mathbb{Z}.\]
Since the image under $\varphi$ of $x^{(y^m)}$ is contained in
$\mathfrak{fg}\langle a,bc\rangle_a\sube \mathfrak{fg}\langle a,b\rangle$ and the image of
$z^{(y^n)}$ is contained in $\mathfrak{fg}\langle bc,d\rangle_d\sube 
\mathfrak{fg}\langle c,d\rangle$,
it is clear that $N\sube \ker\varphi$. We will prove the equality of these
two normal subgroups by looking closer at
the structure of $F/N$.

Let $\cmpl{\phantom{x}}\colon F\to F/N$ denote the natural epimorphism.
Since $N\sube \ker\varphi$, the map $\varphi$ factors through $F/N$ and there
is an epimorphism $\cmpl{\varphi}\colon \cmpl{F}\to S$ given by
$\cmpl{\varphi}(\cmpl{x}) = a$,  $\cmpl{\varphi}(\cmpl{y}) = bc$ and
$\cmpl{\varphi}(\cmpl{z}) = d$. Note that $\cmpl{\varphi}$ maps $\langle 
\cx,\cy\rangle$ onto $\langle a,bc\rangle$ and the latter group is free
on the two generators. Thus $\langle\cx,\cy\rangle$ is also free and
$\cp\colon \fg\langle \cx,\cy\rangle\to \fg\langle a,bc\rangle$ is an isomorphism. In particular,
$\cp\colon \fg\langle \cx,\cy\rangle_{\cx} \to \fg\langle a,bc\rangle_a$ is also an isomorphism.
Note that $\fg\langle \cx,\cy\rangle_{\cx}$ is normalized by $\cx$ and $\cy$. Furthermore
it is centralized by $\cz$ since $g(m,0)\in N$ for all $m\in\mathbb{Z}$. Thus
$\fg\langle \cx,\cy\rangle_{\cx}$ is normal in $\cmpl F$ and similarly so is
$\fg\langle \cy,\cz\rangle_{\cz}$. Note that the relations $g(m,n)\in N$ imply that
these two normal subgroups commute elementwise. Hence
$\cmpl{F}_0 = \fg\langle \cx,\cy\rangle_{\cx}\cdot \fg\langle \cy,\cz\rangle_{\cz}$ is the direct product
of the two normal subgroups and we conclude easily that
$\cp\colon \cmpl{F}_0\to S_0$ is an isomorphism.

Finally note that $S/S_0$ is the infinite cyclic group generated by the
image of $bc$, and $\cmpl{F}/\cmpl{F}_0$ is cyclic, generated by the image
of $\cy$.  With this, we conclude that $\cp$ is an isomorphism and
hence $N=\ker\varphi$. In other words, $N$ is the normal subgroup of relations
of $S$.
\end{proof}

\begin{step3}
$N$ is not finitely generated as a normal subgroup of $F$
and hence $S$ is not finitely presented.
\end{step3}

\begin{proof}
If $N$ is finitely generated as a normal subgroup, then it is generated as a normal subgroup by finitely many of the elements
$g(m,n)$. Choose an integer $r$ larger than this number of $g(m,n)$ generators. Then we have less than $r$ remainders
$n-m\bmod r$ determined by these generators and so there is at least one remainder, say $k$, that is missing.
In particular, $N$ is generated as a normal subgroup
by all the elements $g(m,n)$ that satisfy $n-m\not\equiv k\bmod r$.
The goal is to show that these relations do not imply the remaining ones,
and for this we need some sort of wreath product structure.

Consider the symmetric group $G$ on $\{ 0,1,\ldots, r-1\}$ and with two additional symbols $*$ and $\bullet$. Let $u$ and $w$ be the two transpositions
$u= (*,k)$ and $w=(\bullet,0)$ and let $v$ be the $r$-cycle
$v=(0,1,\ldots, r-1)$. Since $u^{(v^m)}= (*, m+k \bmod r)$ and
$w^{(v^n)} = (\bullet, n\bmod r)$, and since transpositions commute
if and only if they are identical or disjoint, we see that
$u^{(v^m)}$ commutes with
$w^{(v^n)}$ if and only if $m+k\not\equiv n\bmod r$ or equivalently
$n-m \not\equiv k \bmod r$. In particular, if $\theta$ is the homomorphism
$\theta\colon F\to G$ given by $\theta(x)=u$, $\theta(y)=v$ and
$\theta(z)=w$, then $\ker\theta\nor F$ contains $g(m,n)$
if and only if $n-m\not\equiv k \bmod r$. Thus those $g(m,n)$ in $\ker\theta$
cannot generate all the $g(m,n)$ as a normal subgroup of $F$. With this 
contradiction, we
conclude from Lemma~1.1 that $S$ is not finitely presented.
\end{proof}

\section{Lie Algebras and their Enveloping Algebras}

Let $A$ be an associative $K$-algebra and define the map
$[\ ,\ ]\colon A\times A \to A$ by $[a,b]=ab-ba$ for all $a,b\in A$. Then
it is easy to verify that $[\ ,\ ]$ is bilinear, skew-symmetric and
satisfies the Jacobi identity. Thus, in this way, the elements of $A$
form a Lie algebra which we denote by $\mathfrak{L}(A)$.
Note that if $\sigma\colon A\to B$ is an algebra homomorphism, then the same
map determines a Lie homomorphism $\sigma\colon\mathfrak{L}(A) \to\mathfrak{L}(B)$. Of course, if $L$ is an arbitrary Lie algebra, then there is no
reason to believe that $L$ is equal to some $\mathfrak{L}(A)$. However it
is true that each such $L$ is a Lie subalgebra of some $\mathfrak{L}(A)$
and in some sense, the largest choice of $A$, with $A$ generated by $L$,
is the universal enveloping algebra $U(L)$.

The construction of $U(L)$ starts with its universal definition.
Let $L$ be fixed and consider the set of all pairs $(A,\theta)$,
where $A$ is a $K$-algebra and $\theta\colon L\to \mathfrak{L}(A)$ is a Lie
homomorphism. As usual, if $\sigma\colon A\to B$ is an algebra homomorphism, then the composite map $\sigma\theta\colon L\to \mathfrak{L}(B)$ is
a Lie homomorphism and hence $(B,\sigma\theta)$ is an allowable pair. 
A \emph{universal enveloping algebra} for $L$ is therefore defined to be a pair $(U,\theta)$ such that, for any other pair $(B,\phi)$, there exists a
unique algebra homomorphism $\sigma\colon U\to B$ with $\phi=\sigma\theta$.
It is fairly easy to prove that $(U,\theta)$ exists and that it is unique up to
suitable isomorphism. Unfortunately, the existence proof does not tell us what
$U$ really looks like. In particular, without a good deal of work,
it does not settle the question of whether $\theta\colon L\to \mathfrak{L}
(U)$ is one-to-one. In fact, $\theta$ is one-to-one and this is the important
Poincar\'e-Birkhoff-Witt Theorem (see \cite[\S V.2]{J} for more details).

If $X$ is any set of elements, let $F_X$ be the free $K$-algebra
on the variables $X$ and let $L_X$ be the Lie subalgebra of
$\mathfrak{L}(F_X)$ generated by $X$. Then the various $L_X$'s play the role of the free Lie algebras in this theory, and with this, we can speak about finitely
presented Lie algebras. For example, suppose $\cmpl L$ is a Lie
algebra generated by the set $\cmpl X$ and let $\sigma\colon X\to\cmpl X$
be a one-to-one correspondence of sets. Then $\cmpl X$ generates
$U(\cmpl L)$ as an algebra, and $\sigma$ extends to an epimorphism
$\sigma\colon F_X \to U(\cmpl L)$. The restriction of $\sigma$
then yields a Lie epimorphism of $L_X$ onto $\cmpl L$. It follows
easily that $U(L_X)=F_X$ and thus the following result of \cite{Bm},
the Lie analog of Theorem~2.3,
comes as no surprise.

\begin{theorem}
Let $L$ be a Lie algebra over the field $K$.
Then $L$ is finitely presented as a Lie algebra if and only if
$U(L)$ is finitely presented as an associative $K$-algebra.
\end{theorem}

Next, we consider the Lie algebra variant of Lemma 2.4 using similar but somewhat easier arguments.

\begin{lemma}
Let $K$ be a field, let $L$ be a finite-dimensional $K$-Lie algebra, and let $U=U(L)$ be its enveloping
algebra. Suppose that either $\ch K\neq 2$ and $L\neq 0$ or $\ch K=2$ and $L$ is not commutative.
Then the Lie algebra $U\rtimes L$ is finitely generated but not finitely presented. Here
$U$ is viewed as a commutative Lie algebra and the ad action of $\ell\in L$ on $U$ is given by
left multiplication.
\end{lemma}

\begin{proof}
We consider the ring $S=U_3$ of $3\times 3$ matrices over $U$ and we write $\alpha *\beta=\alpha\beta -
\beta\alpha$ for all $\alpha,\beta\in S$. Of course, $S$ becomes a Lie algebra under $*$ and, as is to be expected, we
consider a certain Lie subalgebra. Recall that $U$ has an antipode $\sigma$, so that $\sigma$
is an antiautomorphism that maps each element of $L$ to its negative. Now for any $a,b\in U$ and
$\ell \in L$ write
\[ [\ell, a, b] = \begin{bmatrix} 0 &a^\sigma& b\\ & \ell & a\\ & & 0 \end{bmatrix}\in S\]
and $\cmpl\ell =[\ell,0,0]$.
Notice that $\cmpl L=\lset \cmpl\ell\mid \ell\in L\rset$ is a Lie algebra under $*$, naturally isomorphic to $L$,
and that $\cmpl L$ normalizes the set $\mathcal{L} =\lset [0,a,b]\mid a,b\in U\rset$ since
\begin{equation} \tag{{\bf ad}}
[\ell,0,0] * [0,a,b] =\begin{bmatrix} 0& -a^\sigma \ell& 0\\ & 0& \ell a\\ && 0\end{bmatrix}
\end{equation}
and $(\ell a)^\sigma = a^\sigma \ell^\sigma = -a^\sigma \ell$.

Next, we see that $\mathcal{L}$ is a Lie subalgebra of $S$ since
\begin{equation} \tag{{\bf comm}}
[0,a,b]*[0,c,d]=[0,0, a^\sigma c-c^\sigma a]\in\mathcal{L}.
\end{equation}
Indeed, the latter element is in $\mathcal{Z}\sube\mathcal{L}$ where $\mathcal{Z}=\lset [0,0,e]\mid e\in U\rset$
is central in $\mathcal{L}+\cmpl L=  \mathcal{L}\rtimes \cmpl L$.

Finally, let $\mathcal{W}$ be the finitely generated Lie subalgebra of $\mathcal{L}\rtimes \cmpl L$ generated by
$\cmpl L$ and $[0,1,0]\in\mathcal{L}$. Since $U(L)$ is generated by $L$, as a $K$-algebra, it follows from equation ({\bf ad}) that 
$\mathcal{W}$ contains all $[0,a,0]$ with $a\in U$. Hence, by ({\bf comm}), $\mathcal{W}\cap \mathcal{Z}$
contains all $[0,0,a^\sigma c-c^\sigma a]$ with $a,c\in U$.  Suppose first that $\ch k\neq 2$ and $L\neq 0$, and
choose $0\neq \ell\in L$.
Then, by taking $a=1$ and $c$ any odd power of $\ell$,
we see that $\mathcal{W}\cap \mathcal{Z}$ contains
all $[0,0,d]$ with $d$ an odd power of $\ell$.  On the other hand, if $\ch K=2$, then $\sigma$ fixes all elements of $L$
and hence all powers of elements of $L$. In particular, by taking $a,\ell\in L$ that do not commute and $c=\ell^n$
for any odd integer $n$, we see that $d=a^\sigma c-c^\sigma a= a\ell^n-\ell^n a= \ell^{n-1} m$, where $m\in L$ is the nonzero
Lie product of $a$ and $\ell$.

Hence, in both cases, it follows that $\mathcal{W}\cap \mathcal{Z}$ is a non-finitely generated central ideal
of $\mathcal{W}$ and, since $\mathcal{W}$ is finitely generated, Lemma~1.1 implies that $\mathcal{W}/
(\mathcal{W}\cap\mathcal{Z})$ is a finitely generated but not finitely presented Lie algebra.
A close look at equations ({\bf ad}) and ({\bf comm}) show that $\mathcal{W}/
(\mathcal{W}\cap\mathcal{Z})$  is isomorphic to $U\rtimes L$ where $U$ is viewed as an abelian Lie algebra and $L$
acts on $U$ via left multiplication.
\end{proof}

The special case where $L=Kx$ is 1-dimensional is of interest. Here $U(L) =K[x]$ is the polynomial ring in $x$
so $U$ has $K$-basis $a_i=x^i$ for $i=0,1,2,...$. Then $x$ acts on $U$ via the derivation $d(a_i) =x a_i= a_{i+1}$.
This is an example due to Bahturin \cite{Bh}, proved using a result of  Bryant and Groves \cite{BG}.

Let $K$ be a field. A $K$-Lie algebra
is said to be \emph{coherent} if every finitely generated sub-Lie algebra
is finitely presented. 
The following example, due to Roos in \cite[page 461]{Ro},
shows that finitely presented Lie algebras are not necessarily coherent. We offer an elementary proof here 
analogous to that of Theorem 2.5 and different from the original
homological argument. 
For convenience, we write $\mathfrak{fl}\langle x,y\rangle$ for the free $K$-Lie algebra on the two generators.

\begin{theorem}
Let $T=\mathfrak{fl}\langle a,b\rangle \oplus \mathfrak{fl}\langle c,d\rangle $ be the direct sum of the two free $K$-Lie algebras
and let $R$ be the three generator sub-$K$-Lie algebra of $T$ that is
given by $R=\langle a,b+c,d\rangle$.
Then $R\nor T$ with $T/R$ one dimensional. Furthermore, $T$ is finitely presented, but $R$ is not. In particular, $T$ is not coherent.
\end{theorem}

\begin{proof}
Since finitely generated free Lie algebras are finitely presented, 
we know that $T$ is
finitely presented. We now 
proceed in a series of three steps.
\end{proof}

\begin{step1} The structure of $R$.
\end{step1}

\begin{proof}
As above, we write $\mathfrak{fl}\langle x,y\rangle $ for the free 
$K$-Lie algebra on the two generators.
Then we have an epimomorphism $\mathfrak{fl}\langle x,y\rangle  \to Ky$
given by $x\mapsto 0$, $y\mapsto y$ and we denote its kernel by 
$\mathfrak{fl}\langle x,y\rangle_x$. This kernel is the ideal of $\mathfrak{fl}\langle x,y\rangle $
generated as a Lie algebra by all elements $x\mdot (\ad y)^n$ for $n\in \mathbb{Z}^+$, the nonnegative integers.
Indeed, the sub-Lie algebra of $\fl\la x,y\ra$ generated by all these 
$x\mdot (\ad y)^n$ is clearly normalized by $x$ and $y$, and when one mods
out by this sub-Lie algebra, only $y$ remains.
Similarly, we write $\mathfrak{fl}\langle x,y\rangle_y$ for the sub-Lie algebra of 
$\mathfrak{fl}\langle x,y\rangle $ generated by all the elements $y\mdot (\ad x)^n$ with 
$n\in\mathbb{Z}^+$.  

Now the projection of $R$ onto the first factor of $T$ is given
by $a\mapsto a$, $b+c\mapsto b$ and $d\mapsto 0$. Thus
$\langle a,b+c\rangle$ is free of rank 2, and similarly $\langle b+c,d\rangle$
is also free of rank 2. Since $c$ commutes with $a$ and $b$, we see,
by induction on $n$, that
$a\mdot{(\ad (b+c))^n} = a\mdot{(\ad b)^n}$ and hence $\mathfrak{fl}\langle a,b+c\rangle_a =
\mathfrak{fl}\langle a,b\rangle_a\sube \mathfrak{fl}\langle a,b\rangle$. In other words, the latter two ``sub $a$'' Lie algebras are
identical as subalgebras of $T$. Note that the first formulation 
shows that the subalgebra is in $R$
and is normalized by $a$ and $b+c$, while the second formulation 
shows that it is centralized
by $c$ and $d$. Thus this Lie subalgebra is an ideal in $T$. Similarly we have
$\mathfrak{fl}\langle b+c,d\rangle_d = \mathfrak{fl}\langle c,d\rangle_d\sube \mathfrak{fl}\langle c,d\rangle$
is also an ideal of $T$.
Set $R_0 = \mathfrak{fl}\langle a,b\rangle_a+ \mathfrak{fl}\langle c,d\rangle_d\sube 
\mathfrak{fl}\langle a,b\rangle\oplus \mathfrak{fl}\langle c,d\rangle$ so $R_0\sube R$, $R_0\nor T$
and $T/R_0$ is naturally isomorphic to the two dimensional commutative algebra
$Kb+Kc$. Furthermore, $R/R_0$ corresponds to
the subspace $K(b+c)$, so $R/R_0$ is one dimensional. 
In addition, $R\nor T$ and $T/R$ is
one dimensional.
\end{proof}

\begin{step2}
The relations of $R$.
\end{step2}

\begin{proof}
Let $F=\mathfrak{fl}\langle x,y,z\rangle$ be the free Lie algebra on generators $x,y,z$
and consider the epimorphism $\varphi\colon F\to R$ given by
$\varphi(x)=a,\ \varphi(y)=b+c$ and $\varphi(z)=d$. We will precisely determine
the kernel of $\varphi$. To this end, let $I$ be the ideal 
of $F$ generated, as a Lie ideal, by the commutators
\[ h(m,n)=[ x\mdot{(\ad y)^m}, z\mdot{(\ad y)^n} ] \qquad \mathrm{for\ all}\quad m,n\in\mathbb{Z}^+.\]
Since the image under $\varphi$ of $x\mdot{(\ad y)^m}$ is contained in
$\mathfrak{fl}\langle a,b+c\rangle_a\sube \mathfrak{fl}\langle a,b\rangle$ and the image of
$z\mdot{(\ad y)^n}$ is contained in $\mathfrak{fl}\langle b+c,d\rangle_d\sube 
\mathfrak{fl}\langle c,d\rangle$,
it is clear that $I\sube \ker\varphi$. We will prove the equality of these
two ideals by looking closer at
the structure of the Lie algebra $F/I$.

Let $\cmpl{\phantom{x}}\colon F\to F/I$ denote the natural epimorphism.
Since $I\sube \ker\varphi$, the map $\varphi$ factors through $F/I$ and there
is an epimorphism $\cmpl{\varphi}\colon \cmpl{F}\to R$ given by
$\cmpl{\varphi}(\cmpl{x}) = a$,  $\cmpl{\varphi}(\cmpl{y}) = b+c$ and
$\cmpl{\varphi}(\cmpl{z}) = d$. Note that $\cmpl{\varphi}$ maps $\langle 
\cx,\cy\rangle$ onto $\langle a,b+c\rangle$ and the latter Lie algebra is free
on the two generators. Thus $\langle\cx,\cy\rangle$ is also free and
$\cp\colon \fl\langle \cx,\cy\rangle\to \fl\langle a,b+c\rangle$ is an isomorphism. In particular,
$\cp\colon \fl\langle \cx,\cy\rangle_{\cx} \to \fl\langle a,b+c\rangle_a$ is also an isomorphism.
Note that $\fl\langle \cx,\cy\rangle_{\cx}$ is normalized by $\cx$ and $\cy$. Furthermore
it is centralized by $\cz$ since $h(m,0)\in I$ for all $m\in\mathbb{Z}^+$. Thus
$\fl\langle \cx,\cy\rangle_{\cx}$ is an ideal in $\cmpl F$ and similarly so is
$\fl\langle \cy,\cz\rangle_{\cz}$. Note that the relations $h(m,n)\in I$ imply that
these two ideals commute elementwise. Hence
$\cmpl{F}_0 = \fl\langle \cx,\cy\rangle_{\cx}+ \fl\langle \cy,\cz\rangle_{\cz}$ is the direct sum
of the two ideals and we conclude easily that
$\cp\colon \cmpl{F}_0\to R_0$ is an isomorphism.

Finally note that $R/R_0$ is one dimensional generated by the
image of $b+c$, and $\cmpl{F}/\cmpl{F}_0$ is generated by the image
of $\cy$.  With this, we conclude that $\cp$ is an isomorphism and
hence $I=\ker\varphi$. Thus, $I$ is the ideal of relations
of $R$.
\end{proof}

\begin{step3}
$I$ is not finitely generated as a Lie ideal of $F$
and hence $R$ is not finitely presented.
\end{step3}

\begin{proof}
If $I$ is finitely generated as an ideal of $F$, then 
it is generated as a Lie ideal by finitely many of the elements
$h(m,n)$. Choose an integer $s$ larger than 
the sums $m+n$ for these finitely many generators of $I$ and let $M$ denote the ring
of $(s+3)\times (s+3)$ matrices over $K$. We label the rows and columns
of $M$ by the numbers $\lset 0,1,\ldots, s\rset$ and the additional symbols
$*$ and $\bullet$. In addition, to avoid subscripts, we let
$e(i,j)$ denote the usual matrix units. Now set
$u=e(*,0)$, $w= e(s,\bullet)$ and $v=\sum_{i=0}^{s-1} e(i,i+1)$.

Working in the Lie algebra $\mathcal{L}(M)$ we see that
$u\mdot (\ad v)^m = e(*,m)$ and $w\mdot (\ad v)^n = \pm e(s-n,\bullet)$
for $0\leq m,n\leq s$.
In particular, if $\theta$ is the Lie homomorphism
$\theta\colon F\to \mathcal{L}(M)$ given by $\theta(x)=u$, $\theta(y)=v$ and
$\theta(z)=w$, then $\ker\theta\nor F$ contains those $h(m,n)$
with $m+n<s$ since $e(*,i)$ commutes with $e(j,\bullet)$ if
and only if $i\neq j$. But note that $\ker\theta$ does not contain $h(0,s)$,
so those $h(m,n)$ with $m+n<s$ cannot generate all $h(m,n)$
as an ideal of $F$.
With this 
contradiction, we
conclude from Lemma~1.1 that $R$ is not finitely presented.
\end{proof}

\section{Rings and Algebras}

Suppose $R\supe S$ are rings with the same 1. If $R$ is a finitely generated right or left
$S$-module and if $S$ is finitely generated as a ring, then $R$ is certainly also finitely generated
as a ring. Thus it is natural to ask whether the finitely presented property lifts from $S$ to $R$. As we see
below, the answer is ``no''.

\begin{lemma}
There exist rings $R\supe S$ and a central idempotent $e$ of $R$ such that $R=Se+S(1-e)$. Furthermore,
$S$ is finitely presented, but $R$ is not.
\end{lemma}

\begin{proof}
Let $T$ be a finitely presented ring with a homomorphic image $\cmpl T$ that is not finitely
presented. For example, $\cmpl T$ could be any finitely generated ring that is not finitely
presented and $T$ could be a finitely generated free ring that maps onto $\cmpl T$. Now set
$R=T\oplus \cmpl T$, so that $R$ in not finitely presented by Lemma 1.3. Furthermore, note
that $e=1\oplus 0$ is a central idempotent in $R$.
Now embed $T$ into $R$ via the map $t\mapsto t\oplus \cmpl t$ and let $S$ denote the image
of $T$. Then $R=Se+S(1-e)$ and $S\cong T$ is finitely presented.
\end{proof}

If $R$ is a right or left Noetherian ring, then all two-sided ideals of $R$ are certainly
finitely generated. Thus Lemma 1.2 implies

\begin{lemma}
Let $R$ be a right or left Noetherian ring that is finitely presented. Then all homomorphic
images of $R$ are also finitely presented.
\end{lemma}

As a consequence, we have the well-known fact that every
finitely generated commutative ring is finitely presented.
In view of this, it is natural to ask whether every affine (that is, finitely generated)
PI-algebra is necessarily finitely presented. Using a slightly simpler variant of an example constructed in \cite{SW}, 
but with the same proof, we show below that
the answer is ``no''.

\begin{lemma} Let $F$ be a field. There exists a finitely generated prime $F$-algebra,
satisfying the multilinear identities of $2\times 2$ matrices, that is not finitely presented.
\end{lemma}

\begin{proof}
Following \cite{SW}, let $A=F[x,y,z]$ be the polynomial ring over $F$ in the
three commuting indeterminates. Furthermore, let $I=yA+zA$ and let $R\sube \mathbf{M}_2(A)$ be given by
\[ R =\begin{bmatrix} F+I^2 & I\\ I & A \end{bmatrix}.\]
Then $R$ satisfies the identities of $\mathbf{M}_2(A)$ and it is easy to see that $R$ is a finitely
generated $F$-algebra. Furthermore, if $P$ is a nonzero principal prime ideal of $A$ contained in $I^2$,
then 
\[ Q= \begin{bmatrix} P&P\\P&P \end{bmatrix}\]
is a prime ideal of $R$ that is not finitely generated as a 2-sided ideal. Thus Lemma~1.1
implies that the ring $R/Q$ is a finitely generated prime PI-algebra that is not finitely presented.
\end{proof}

The $K$-algebras $A$ and $B$ are said to be \emph{Morita equivalent}
if their categories of right modules are equivalent. It is known that
this occurs if and only if $B\cong e M_n(A) e$
where $e$ is an idempotent of the matrix ring $M_n(A)$ satisfying
\[M_n(A) e M_n(A) = M_n(A).\] See \cite[Proposition 18.33]{L}. Such
idempotents $e$ are said to be \emph{full}. According to \cite{MS},
the property of being a finitely generated $K$-algebra  is a Morita invariant,
namely it carries over from an algebra to a Morita equivalent one.
As a generalization, we have the result of \cite{AA2}, which asserts that the property
of being finitely presented is also a Morita invariant.
The proof of this result is clever but simple, so we sketch it below. We start
with \cite[Theorem 1]{AA1}.

\begin{proposition}
If $A=A_0\oplus A_1$ is a $\mathbb{Z}_2$-graded $K$-algebra that is
finitely presented and if $A_1^2=A_0$, then the $K$-subalgebra $A_0$ is also
finitely presented.
\end{proposition}

\begin{proof}
Since $A_1^2=A_0$, it follows that $A$ is finitely generated as an
algebra with generators in $A_1$. Thus we can construct a $K$-algebra epimorphism $\theta$ from
the finitely generated free algebra $F$ in the variables $X=\{ x_1,x_2,\ldots, x_n\}$ onto $A$ with each generator
mapping to $A_1$. Note that $F$ is also $\mathbb{Z}_2$-graded with
$F_0$ spanned by all monomials of even length and $F_1$ spanned by all
monomials of odd length. Clearly $\theta$ preserves the grading, so
$\theta(F_0)= A_0$. Furthermore, $F_0$ is a finitely generated free
algebra,
with variables $x_ix_j$, so the restriction $\theta_0\colon F_0\to
A_0$
is a free presentation of $A_0$. It remains to show that $I_0=\ker
\theta_0$
is a finitely generated ideal of $F_0$.

Since $A$ is finitely presented and $\theta$ preserves the grading,
it is clear that $I=\ker\theta$ is generated by a finite set
$M_0\cup M_1$ of homogeneous elements. Then we have
\[ I = (F_0\oplus F_1)(M_0\cup M_1)(F_0\oplus F_1)\]
and $F_1=F_0X = X F_0$, so we conclude that $I_0=I\cap F_0$ is generated by
the finite set $M_0\cup M_1X \cup X M_1\cup XM_0X$, as required.
\end{proof}

Now for the promised result \cite[Theorem 2]{AA1} and \cite[Theorem 1]{AA2}.

\begin{theorem}\label{AA}
The property of being a finitely presented $K$-algebra is
a Morita invariant.
\end{theorem}

\begin{proof}
If $A$ is a finitely presented algebra, then so is any matrix ring
${M}_n(A)$.
Thus we need only show that if $A$ is finitely presented and $e$ is a
full idempotent of $A$, then $eAe$ is finitely presented. 

Suppose first that both $e$ and $1-e$ are full. Then
$A$ is $\mathbb{Z}_2$-graded with $A_0=eAe\oplus (1-e)A(1-e)$
and $A_1=eA(1-e)\oplus (1-e)Ae$. Furthermore, since $e$ and $1-e$ are
full, it follows that $A_1^2=A_0$. Thus the previous proposition implies that
$A_0$ is finitely presented and hence so is its direct summand $eAe$.

Finally, suppose only that $e$ is full and
set $\cmpl{A}={M}_2(A)$ and $\cmpl{e} = \begin{bmatrix} e&0 \\
0& 0 \end{bmatrix}\in\cmpl{A}$. Then $\cmpl{A}$ is finitely presented
and $\cmpl{e}$ is a full idempotent of $\cmpl{A}$ with $1-\cmpl{e}$ also full.
Thus by the above,
\[ \cmpl{e}\cmpl{A}\cmpl{e} =\begin{bmatrix} eAe &0\\ 0&0\end{bmatrix}
\cong eAe\]
is indeed finitely presented.
\end{proof}

Now let $V$ be a right $A$-module. Then $V$ is said to be
\emph{finitely presented} if $V\cong F/U$ where $F$ is a finitely
generated free $A$-module and $U$ is a finitely generated submodule
of $F$. The next result follows easily from
Schanuel's Lemma (see \cite[Theorem 1, page 161]{K}).

\begin{lemma} Let $V$ be a finitely presented right $A$-module,
let $F'$ be a finitely generated free $A$-module and let
$\phi\colon F'\to V$ be an epimorphism. Then $U'=\ker\phi$
is a finitely generated submodule of $F'$ and hence $V\cong F'/U'$
is also a finite presentation for the module $V$.
\end{lemma}

\begin{proof}
Since $0\to U\to F\to V$ and $0\to U'\to F'\to V$ are exact,
Schanuel's Lemma implies that $F\oplus U'\cong F'\oplus U$. Indeed,
this holds even if $F$ and $F'$ are merely assumed to be projective.
In particular, $U'$ is a homomorphic image of the finitely generated
module $F'\oplus U$.
\end{proof}

Of course, if $A$ is right Noetherian, then every finitely generated
right $A$-module is finitely presented. For convenience, we introduce a
concrete realization of these modules. Specifically, suppose
$c_1,c_2,\ldots, c_n$ generate the $A$-module $V$ and let
$F= A^n = (A,A,\ldots, A)$ be a free $A$-module of rank $n$.
Then we have an epimorphism $\phi\colon F\to V$ given by
$\phi(a_1,a_2,\ldots, a_n) =\sum_i c_i a_i$. Clearly, \[U=\ker\phi
=\{ (a_1,a_2,\ldots, a_n)\mid \sum_i c_i a_i=0\} \]
and we will think of this as the relation module for $V$ corresponding
to the generators $c_1,c_2,\ldots, c_n$.

\begin{theorem}
Let $A$ be a finitely presented $K$-algebra and let $I$ be a right
ideal of $A$ that is finitely presented as a right $A$-module.
If $AI=A$, then $K+I$ is a finitely presented $K$-algebra.
\end{theorem}

\begin{proof}
If $I=A$, the result is trivial, so we can assume that $I$ is proper.
Note that $AI=A$ implies that $I\neq 0$. Thus by assumption
$I$ has $n\geq 1$ generators $c_1,c_2,
\ldots, c_n$ with relation module generated by the finitely many relations
$g_j=(a_{1j}, a_{2j},\ldots, a_{nj})$.
We fix this notation throughout.

Now we look inside the $2\times 2$ matrix ring over $A$ and consider the
subalgebra
\[ M=\begin{bmatrix} K+I & I\\ A & A \end{bmatrix}.\]
In some sense our goal is to show that $M$ is finitely presented
and we proceed in a series of steps.
\end{proof}

\begin{step1}
The construction of $B$, a finitely presented algebra that mirrors $M$.
\end{step1}

\begin{proof}
First $B$ is a $K$-algebra with 1, and we need the idempotents
corresponding to $e_{11}$ and $e_{22}$. So $B$ has generators $e$
and $f$ with relations that assert that $e$ and $f$ are orthogonal
idempotents that sum to 1.

Next we adjoin the lower right corner of $B$ corresponding to the
idempotent $f$. Specifically, we construct a monomorphism
$\cmpl{\phantom{x}}\colon A \to fBf$ with $\cmpl{1}=f$.
Clearly $\cmpl{A}$ is the $K$-subalgebra of $B$
generated by the images $\cmpl{a}$ of the finitely many generators
$a\in A$ and subject to their finitely many relations. Furthermore,
we insist that for each generator $a\in A$, its image $\cmpl{a}$
belongs to $fBf$ and this is achieved by adding the relations
$f\cmpl{a} =\cmpl{a}=\cmpl{a}f$.

We need just  a few more generators and relations. First $r\in B$ plays the
role of $e_{21}$, so we have $fr=r=re$. Further, for $i=1,2,\ldots, n$,
$s_i\in B$ plays the role of $c_i e_{12} $, so $es_i=s_i= s_if$
and $rs_i=\cmpl{c}_i\in \cmpl{A}$, where in the latter relations, the elements
$\cmpl{c}_i$ are written explicitly in terms of the generators of $\cmpl{A}$.
Finally, for each $g_j=(a_{1j}, a_{2j},\ldots, a_{nj})$, we adjoin the
relation \[\widetilde g_j= s_1 \cmpl{a}_{1j}+ s_2 \cmpl{a}_{2j}+\cdots+ s_n \cmpl{a}_{nj}=0\]
where again, each $\cmpl{a}_{ij}$ is written explicitly in terms of the
generators of $\cmpl{A}$.

Thus $B$, as defined above is a finitely presented $K$-algebra
and we have an algebra homomorphism $\theta\colon B\to M$ given by
$e \mapsto e_{11}$, $f\mapsto e_{22}$, $\cmpl{a}\mapsto a e_{22}$
for all $a\in A$, $r\mapsto e_{21}$ and $s_i\mapsto c_i e_{12}$
for $i=1,2,\ldots,n$. Note that, for each $j$,
\begin{align*}
\theta(\widetilde g_j)&= c_1e_{12} {a}_{1j}e_{22}+ c_2e_{12} {a}_{2j}e_{22}+\cdots+ c_n e_{12} {a}_{nj}e_{22}\\
&= (c_1 {a}_{1j}+ c_2{a}_{2j}+\cdots+ c_n {a}_{nj})e_{12} =0,
\end{align*}
as required.
We will show below that $\theta$ is an isomorphism.
\end{proof}

\begin{step2}
$B$ is a direct sum of its four corners
\[ B= eBe \oplus eBf\oplus fBe\oplus fBf\]
with
$fBf = \cmpl{A}$,
$fBe = \cmpl{A}r$,
$eBf = \sum_i s_i \cmpl{A}$ and
$eBe = Ke + \sum_i s_i \cmpl{A} r$.
\end{step2}

\begin{proof}
The first formula follows from
\[ B = 1B1 = (e+f)B(e+f)\]
and the four summands here are called the corners of $B$.
Notice that a product of corners is either 0 or contained in
another corner. In particular, since all generators of $B$
(other than 1) are contained in corners, we can determine
the corners of $B$ by considering which products of generators they
contain.

We remark that in the argument below, each product $\pi'$
is either empty and hence equal to 1 or necessarily in a corner
that is easy to describe from context.

We start with $fBf$ and, by construction, we know that $fBf\supe
\cmpl{A}$. Now suppose that $\pi\in fBf$ is a nonzero product of generators.
We show by induction on the number of factors
in $\pi$ that $\pi\in \cmpl{A}$. Now $f\pi=\pi \neq 0$
so the left-most factor of $\pi$ is $f$, $r$ or $\cmpl{a}$ for
some generator $a\in A$. In the first and third cases we have
$\pi = f \pi'$ or $\cmpl{a} \pi'$ where $\pi'\in fBf$ is a shorter
product. By induction, $\pi'\in\cmpl{A}$ and then
$\pi\in\cmpl{A}$. On the other hand, if $\pi$ starts with $r$, then
this factor can only be followed by several $e$'s and then some
$s_i$. But then $r e^t s_i = \cmpl{c}_i$, so $\pi =\cmpl{c}_i \pi'$
and again $\pi\in \cmpl{A}$. We conclude that $fBf = \cmpl{A}$.

Next note that $eBf\supe \sum_i s_i\cmpl{A}$. Conversely suppose
$\pi\in eBf$ is a nonzero product of generators. Then $e\pi=\pi
\neq 0$ 
so $\pi$ starts on the left with $e$ or some $s_i$. In the former
case, $\pi=e \pi'$ and $\pi'$ is shorter. Thus, by induction,
$\pi'\in \sum_i s_i\cmpl{A}$ and hence $\pi=e\pi'\in
\sum_i s_i\cmpl{A}$. On the other hand, if $\pi= s_i\pi'$,
then $\pi'\in fBf=\cmpl{A}$ and hence $\pi\in s_i\cmpl{A}$.

The argument for $fBe$ is similar except that we look at the right-most
factor of $\pi$. Indeed, $fBe\supe \cmpl{A} r$ and if $\pi\in fBe$
is a nonzero product of generators, then $\pi e = \pi\neq 0$,
so the right-most factor of $\pi$ is either $e$ or $r$. Thus
either $\pi =\pi' e$ with $\pi'$ shorter or $\pi=\pi' r$ with
$\pi'\in fBf =\cmpl{A}$.

Finally, $eBe$ contains $Ke$ and also
\[ eBf{\cdot} fBe = \sum_i s_i \cmpl{A}{\cdot}\cmpl{A} r =\sum_i
s_i\cmpl{A}r.\]
Conversely suppose $\pi\in eBe$ is a nonzero product of generators. If
$\pi$ has just one factor then it is $e$. Otherwise,
$e\pi e=\pi\neq 0$ and ignoring initial and final factors of $e$,
we see that $\pi = s_i\pi' r$ for some $i$. Then $\pi'
\in fBf =\cmpl{A}$, so $\pi\in s_i \cmpl{A} r$.
\end{proof}

\begin{step3}
$\theta\colon B\to M$ is an isomorphism. Hence $M$ is finitely
presented and therefore so is $K+I$.
\end{step3}

\begin{proof}
Notice that $\theta(\cmpl{a}) = a e_{22}$ for all $a\in A$,
so $\theta$ is one-to-one and onto from the $f,f$ corner
$fBf= \cmpl{A}$ of $B$ to the $e_{22},e_{22}$ corner $A e_{22}$
of $M$. Similarly, $\theta(\cmpl{a}r)= a e_{21}$ shows
that $\theta$ is one-to-one and onto from the $f,e$ corner
$fBe =\cmpl{A}r$ of $B$ to the $e_{22}, e_{11}$ corner
$A e_{21}$ of $M$.

Next
\[ \theta(eBf)=\theta(\sum_i s_i\cmpl{A})=e_{12} \sum_i c_iA=
e_{12}I\]
so in this case $\theta$ is onto. To check that it is one-to-one on $eBf$,
suppose $\sum_i s_i\cmpl{a}_i$ maps to 0 for suitable $a_i\in A$.
Then
\[0=\theta(\sum_i s_i\cmpl{a}_i) = e_{12} \sum_i c_i a_i\]
so
$\sum_i c_i a_i=0$. Thus $(a_1,a_2,\ldots, a_n)$ is in the relation module for
$c_1,c_2,\ldots, c_n$ and hence it is an $A$-linear combination of the
generators $g_j$'s. Since 
\[\widetilde g_j= s_1 \cmpl{a}_{1j}+ s_2 \cmpl{a}_{2j}+\cdots+ s_n \cmpl{a}_{nj}=0\]
in $B$, by assumption, it therefore follows that $\sum_i s_i \cmpl{a}_i=0$.

Finally, note that
\[\theta(eBe) = Ke_{11} +\sum_i e_{12} c_i A e_{22}e_{21}=(K+\sum_i
c_iA) e_{11}=(K+I) e_{11}\]
so $\theta$ maps $eBe$ onto the $e_{11}, e_{11}$ corner of $M$.
If $\beta = ke+\sum_i s_i \cmpl{a}_i r$ maps to 0
for some $k\in K$ and $a_i\in A$, then
\[ 0=\theta(\beta)=\theta (ke+\sum_i s_i \cmpl{a}_i r)= (k+\sum_i c_i a_i) e_{11}\]
and hence $\sum_i c_i a_i = -k$. Since $I\neq A$ by our assumption,
it follows that $k=0$ and then $\sum_i c_i a_i =0$.
As above, we conclude that $\sum_i c_i \cmpl{a}_i =0$,
so $\sum_i c_i \cmpl{a}_i r =0$.
We conclude that $\beta=0$, and $\theta$ is
indeed one-to-one on $eBe$.

Thus $\theta$ is an isomorphism, so $M$ is isomorphic to $B$
and hence it is finitely presented. Furthermore, since $AI=A$, we have
\begin{align*}  M e_{11}M &= \begin{bmatrix} K+I& 0\\A&0 \end{bmatrix}
\begin{bmatrix} K+I& I\\0&0 \end{bmatrix}\\
&= \begin{bmatrix} K+I& I\\ A&A \end{bmatrix} =M
\end{align*}
and hence $K+I\cong e_{11}M e_{11}$ is Morita equivalent to $M$.
Theorem~\ref{AA} now implies that $K+I$ is also finitely
presented.
\end{proof}

We now consider two examples of interest. First let $A$ be the Weyl algebra
over the field $K$, so that $A$ is generated by $x$ and $y$ subject
to the relation $[x,y]=xy-yx=1$. Then $A$ is the Ore extension
$K[x][y;d]$, where $d$ is the ordinary derivation on the polynomial ring
$K[x]$. It follows from this formulation that $A$ satisfies:
(1) it is finitely presented, (2) it is right and left Noetherian,
(3) it is a domain, and (4) $y$ is not a unit in the ring.
We actually need a slightly larger ring, namely $A'$, the
localization of $A$ at the powers of $x$. Thus $A'$ is the
Ore extension $K[x,x\inv][y;d]$ and again $A'$ satisfies
the above four properties.

Now let $I' = yA'$ be the right ideal of $A'$
generated by $y$. Then $I'$ is proper since $y$ is not a unit,
and it is a free right module of rank 1 since $A'$ is a domain.
Furthermore, $A'I'= A' yA'$
is an ideal containing $xy$ and $yx$ and hence 1. Thus $A'
I'=A'$ and the preceding theorem implies that the
subalgebra $K+I'=K+ yA'$ is finitely presented.
As we see below, this is not the answer we had hoped for.

Indeed, let $W$ be the Witt Lie algebra so that $W$ has the
standard $K$-basis consisting of
the elements $e_i$, with $i\in \mathbb Z$, subject to the 
commutation relations
$[e_i,e_j]=(i-j) e_{i+j}$, for all $i,j\in\mathbb Z$. If $K$ has characteristic 0, then it is clear that $W$ is finitely generated. But surprisingly,
by \cite[page 506]{Ste}, $W$ is also finitely presented. Thus,
by Theorem~3.1, its universal enveloping algebra $B=U(W)$ is a finitely presented
algebra. This algebra has been extensively studied
(see for example \cite{DS} and \cite{ISi}), and in particular it was shown in
\cite{SiW1} that $B$ is neither right nor left Noetherian. We had hoped to use
Lemma 1.2 to exhibit a specific two-sided ideal of $B=U(W)$ that is not finitely
generated.

To this end, observe that we have a homomorphism $\theta$ from $B$ to
$A'$ given by $e_i\mapsto y x^{i+1}$ for all $i\in\mathbb{Z}$,
and the image of $\theta$ is clearly $K+yA'=K+I'$.
If this image were not finitely presented, then by Lemma 1.2, the kernel
of $\theta$ would be a two-sided ideal of $B$ that is not finitely
generated. Alas, as we have seen, $K+I'$ is finitely presented,
and hence it follows from Lemma 1.1 that $\ker\theta$ is a finitely generated two-sided ideal.

This argument has been generalized in \cite{SiW2}.

\section{Hopf Algebras}

The group ring and enveloping algebras examples of Lemmas 2.4 and 3.1 surely extend to Hopf
algebras. Let $H$ be a Hopf algebra with counit $\ve\colon H\to K$, antipode $S\colon H\to H$
and comultiplication $\Delta\colon H\to H\otimes H$. We consider the ring of $3\times 3$ matrices over $H$ and the 
subring $\mathcal{R}$ of all elements
\[ [h,a,b,c] =\begin{bmatrix} \ve(h) & a & c\\ & h & b\\ & & \ve(h) 
\end{bmatrix} \]
with $h,a,b,c\in H$. For all $h\in H$, set $\cmpl h =[h,0,0,0]$ and note that the map
$h\to \cmpl h$ defines an isomorphism of Hopf algebras. Conversely, the map $\mathcal{R}\to H$
given by $[h,a,b,c]\to h$ is an algebra epimorphism with kernel $\mathcal{K}$ consisting of all elements of
$\mathcal{R}$ of the form $[0,a,b,c]$. We have the following product formulas
\[ \cmpl h\mdot [0,a,b,c] =[ 0, \ve(h) a, hb, \ve(h)c] \]
and
\[ [0,a,b,c]\mdot \cmpl h = [0, ah, b\ve(h), c\ve(h)].\]

Now we extend the inner action of $H$ on $H$ to an action on $\mathcal{K}$ by defining
\[ h * [0, a,b,c] = \sum_{(h)} {\cmpl h}_1\mdot [0,a,b,c] \mdot S({\cmpl h}_2) = [0,a', b', c']\]
where
\begin{align*} a' &= a\mdot \sum_{(h)} \ve{(h_1)} S(h_2) = a\mdot S(h),\\
b'&= \sum_{(h)} h_1\ve(S(h_2)) \mdot b = h\mdot b,\  \mathrm{and}\\
c' &=\sum_{(h)} \ve(h_1) \ve(S(h_2)) \mdot c = \ve(h) \mdot c.
\end{align*}
In particular, $H$ stabilizes 
\[ \mathcal{K}_0 = \lset [0,a,b,c] \mid a,b,c\in H,\ a=S(b)\rset.\]

\section{The Fundamental Problem}

The study of finitely presented algebras is an isolated subject. One proves facts
about such algebras and also decides whether certain examples are
finitely presented or not. The fundamental problem is to use such results in
other fields, namely to obtain theorems that do not have
``finitely presented'' in their statement, but use the concept intrinsically
in their proof. At present, we know of only one such result,
an application to the study of twisted group algebras. We briefly
outline the argument below.

Let $K^t[G]$ denote a twisted group algebra of the multiplicative group $G$
over the field $K$. Then $K^t[G]$ is an associative $K$-algebra with $K$-basis
$\cmpl{G}=\lset \cmpl{x}\mid x\in G\rset$ and with multiplicative defined
distributively by $\cmpl{x}\mdot \cmpl{y}=\tau(x,y) \cmpl{xy}$ for all
$x,y\in G$. Here $\tau\colon G\times G\to K^\bullet$ is the \emph{twisting
function} and, as is well-known, the associative law in $K^t[G]$ is
equivalent to the fact that $\tau$ is a 2-cocycle. Furthermore, we can
always assume that $\cmpl{1}=1$ is the identity element of $K^t[G]$.
With this assumption, the \emph{trace map}
$\tr\colon K^t[G]\to K$ is defined linearly by $\tr\cmpl{x}=0$
if $1\neq x\in G$ and $\tr 1=1$. In other words, $\tr\alpha$ is the identity
coefficient of $\alpha\in K^t[G]$. It is clear that $\tr \alpha\beta=
\tr\beta\alpha$ for all $\alpha,\beta\in K^t[G]$.

Of course, the ordinary group algebra $K[G]$ corresponds to the twisting function
with constant value 1. In \cite{K}, Kaplansky proved that if $e$ is an
idempotent in the complex group algebra $\CC[G]$, then $\tr e$ is real and bounded between 0 and 1. Next he observed that if $\sigma$ is any field automorphism of $\CC$, then $\sigma$ yields a ring automorphism of $\CC[G]$,
so $\sigma(e)$ is also an idempotent. Thus $\sigma(\tr e) = \tr(\sigma(e))$
is also real and it follows that $\tr e$ is a totally real algebraic
number. This led Kaplansky to conjecture that these traces are actually always
rational. 

In \cite{Z}, Zalesski proved the conjecture via a beautiful
two step argument. First, he worked over fields $F$ of characteristic
$p>0$ and used properties of the $p$-power map to show that
$\tr e\in\mathrm{GF}(p)$. Then he considered fields $K$ of characteristic 0,
and he constructed suitable places to fields $F$ of characteristic $p>0$ for 
infinitely many different primes $p$. In this way he obtained homomorphisms from
subrings of $K[G]$ to $F[G]$. Since $\tr e$ always mapped to elements
of the prime subfield of $F$, the Frobenius Density Theorem allowed
him to conclude that $\tr e\in\QQ$.

More recently, the paper \cite{P} considered whether this result also holds for twisted group
algebras. In characteristic $p>0$, Zalesski's argument carries over
almost verbatim and again one concludes that $\tr e\in \mathrm{GF}(p)$.
The difficulty occurs in fields of characteristic 0, when trying to map
subrings of $K^t[G]$ to $F^t[G]$. Indeed, note that the basis $\cmpl{G}$
generates a subgroup $\mathcal{G}$ of the unit group of
$K^t[G]$ and certainly we expect $\mathcal{G}\cap K^\bullet$
to be nontrivial. Thus, when we map a subring $R$ of $K^t[G]$
to $F^t[G]$, we need $R\supe \mathcal{G}$, but we also need
$\mathcal{G}\cap K^{\bullet}$ to not map to 0. Unfortunately, it is easy
to construct examples, using Lemma 2.4, of twisted group algebras where
$\mathcal{G}\cap K^\bullet$ fills up the entire nonzero part
of the field. Thus the best we can do with this part of Zalesski's
argument is to prove the result under the assumption that
$\mathcal{G}\cap K^\bullet$ is a finitely generated group.

So where to we go from here? Write $e = \sum_{i=1}^n k_i \cmpl{x_i}\in K^t[G]$.
Then certainly we can assume that $G=\langle x_1,x_2,\ldots, x_n\rangle$
is finitely generated. Furthermore, the equality $e^2=e$ is equivalent
to a finite number of group relations of the form $x_r x_s=x_t$ and $x_rx_s= x_u x_v$,
along
with formulas relating the coefficients and twisting.
With a little care, (see \cite[Lemma 3]{P}), one can then show that there
is a finitely presented group $H$, a twisted group algebra $K^t[H]$
and an idempotent $f\in K^t[H]$ such that $\tr e = \tr f$.
This allows us to assume that the group $G$ above is also finitely
presented.

Now by changing the basis by factors of $K^\bullet$ if necessary, we can 
clearly assume that $\mathcal{G}$
is generated by $\cmpl{x_1},\cmpl{x_2},\ldots,\cmpl{x_n}$.
Furthermore, note that the map $ k\cmpl{x}\mapsto x$ yields a
homomorphism from $\mathcal{G}$ onto $G$ with kernel $\mathcal{G}\cap
K^\bullet$. Thus since $\mathcal{G}$ is finitely generated and $G$
is finitely presented, Lemma 1.1 implies that $\mathcal{G}\cap K^\bullet$
is finitely generated as a normal subgroup of $\mathcal{G}$.
Thus this central subgroup of $\mathcal{G}$ is finitely generated as a group
and Zalesski's argument can now come into play. This completes
the proof of \cite[Theorem 5(i)]{P}, a result that does not mention
finitely presented groups, but certainly uses them.

\end{document}